\documentclass[review]{elsarticle}

\usepackage{lineno,hyperref,amssymb,amsmath}
\modulolinenumbers[5]

\journal{Expositiones Mathematicae}

\newtheorem{thm}{Theorem}%%[section]

\newtheorem{lem}[thm]{Lemma}

\newtheorem{defn}[thm]{Definition}

\newproof{proof}{Proof}

\newcommand{\abs}[1]{\left\vert #1 \right\vert}

\newcommand{\scp}[2]{ \left\langle #1,\ #2  \right\rangle }
\newcommand{\RE}{ \mathop{\mathrm{Re}} }
\newcommand{\IM}{ \mathop{\mathrm{Im}}  }
\newcommand{\C}{\mathbb{C}}
\newcommand{\Cn}{ \mathbb{C}^n }
\newcommand{\Om}{\Omega = \{z\in\Cn : \rho(z)<0\}}
\newcommand{\dom}{\partial\Omega}

\newcommand{\Lp} { L^{p(\cdot)}(\partial\mathcal{E})}

\begin{document}

\begin{frontmatter}

\title{An example of a space $L^{p(\cdot)}$ on which the Cauchy-Leray-Fantappi\`{e} operator for complex ellipsoid is not bounded}
\tnotetext[t1]{This work is supported by the Russian Science Foundation grant 24-11-00087.}

\author[mymainaddress]{Aleksandr Rotkevich}
\ead{rotkevichas@gmail.com}
\address[mymainaddress]{Department of Mathematical analysis, Mathematics and Mechanics Faculty, St. Petersburg  University, 198504,  Universitetsky prospekt, 28, Peterhof, St. Petersburg, Russia\\
Department of Mathematics and Computer Science, St. Petersburg University,	Russia, 14 line of the VO, 29B, 199178 St. Petersburg, Russia
}

\begin{abstract}
We construct an example of a Lebesgue space with variable exponent on which Cauchy-Leray-Fantappi\`{e} operator associated with a complex ellipsoid is not bounded. This result extends previous counterexamples for the unit ball and demonstrates that the logarithmic continuity condition for the exponent function $p(\cdot)$ is sharp even for non-strictly convex domains. The proof is based on an explicit construction of test functions supported near points where the boundary fails to be strictly convex. 
\end{abstract}


\end{frontmatter}

%\linenumbers

\section{Introduction}

Since the 1980s, Lebesgue spaces with variable exponents have been actively studied. Interest in these spaces arises from their numerous applications in partial differential equations and the calculus of variations. Substantial progress has been achieved in the theory of classical operators in harmonic analysis, in particular, for the Hardy–Littlewood maximal operator and for singular integral Calder\`{o}n--Zygmund kernels.

It is well-known that the centered maximal Hardy--Littlewood operator is bounded on the
Lebesgue space $L^{p(\cdot)}$ if the function $p$ satisfies the logarithmic
continuity condition (also referred to as logarithmic H\"{o}lder continuity). This condition is sharp:  Pick and
Ru$\rm{\check{z}}$i$\rm{\check{c}}$ka in \cite{PR01} provided an example
of exponent $p$ that fails to be logarithmic continuous only at
one point for which the Hardy--Littlewood maximal operator is not
bounded on $L^{p(\cdot)}.$ 

In our earlier work \cite{R17}, we examined the
boundedness of the Cauchy operator on the unit circle and the Cauchy--Leray--Fantappi\`{e} integral operator on the unit ball in Lebesgue spaces with variable exponent. There we provided examples of spaces in which these operators are not bounded. The purpose of this note is to extend these counterexamples tho the broader class of domains with degenerate boundary geometry. The example is constructed for the case of a complex ellipsoid
\begin{equation}
	\mathcal{E} = \{ z = (z_1,\ldots,z_n)\in\C^n : \rho(z)= \abs{z_1}^{2m_1}+\ldots+ \abs{z_n}^{2m_n}<1\},
\end{equation}
where $m_j\geq1$ for every $j=1,\ldots,n.$

Unlike the Euclidean ball, the domain $\mathcal{E}$ is not uniformly convex whenever at least one of the parameters $m_j$ differs from $1.$

\paragraph{Principal notations}

\begin{itemize}
	\item $\Cn$ is the space of $n$ complex variables, $n\geq 2,$ $z =
	(z_1,\ldots, z_n),\ z_j = x_j + i y_j;$
    \item Inner product 
    $\scp{z}{w} =
    z_1w_1+\ldots+z_nw_n,\ z=(z_1,\ldots,z_n),\\
    w=(w_1,\ldots,w_n)\in\Cn;
    $
    \item Vector norm  $|z|=\sqrt{\scp{z}{\bar{z}}},$ $z\in\Cn;$
    %\item $dS_n$ - surface measure on $S_n=\partial B_n.$
    \item The expression $g\lesssim f$
     means that there exists a constant $c>0$ such that $g\leq c f.$
    \item The expression $g\asymp f$ means that there exist constant $c,C>0$ such that $cf\leq g\leq C f.$
%    \item For a function $f$ we define operators of complex an conjugate complex derivative 
%    \begin{equation*}
%    	\partial_j f =\frac{\partial f}{\partial z_j} = \frac{1}{2}\left( \frac{\partial f}{\partial x_j} - i \frac{\partial f}{\partial y_j}\right), \quad \bar\partial_j f = \frac{\partial f}{\partial\bar{z}_j} = \frac{1}{2}\left( \frac{\partial f}{\partial x_j} + i \frac{\partial f}{\partial y_j}\right);
%    \end{equation*}
%    
%    \[
%    \partial f = \sum\limits_{k=1}^{n} \frac{\partial f}{\partial z_k} dz_k,\quad \bar{\partial} f = \sum\limits_{k=1}^{n} \frac{\partial f}{\partial\bar{z}_k} d\bar{z}_k,\quad df=\partial f+  \bar{\partial} f;
%    \]
%    \[
%    \abs{\bar\partial f} = \abs{\bar\partial_1 f} + \ldots+\abs{\bar\partial_n f}
%    \]
%    We use the notation 
%    $
%    \scp{\eta}{w} = \sum\limits_{k=1}^{n} \eta_k w_k
%    $
%    to indicate the action of the differential form $\eta=\sum\limits_{k=1}^{n}\eta_k dz_k$ of type $(1,0)$ on the vector
%    $w\in\Cn.$
\end{itemize}

\section{Boundedness of Cauchy--Leray--Fantappi\`{e} integral operator for complex ellipsoid}

Let  $\Om$ be a bounded domain with a defining function $\rho\in
C^2$ such that $\partial\rho\neq 0$ on $\dom$ and $d\sigma$ be a
surface measure on $\dom.$ Consider a $\sigma-$measurable function
$p:\dom\to[1,\infty].$ This function defines the generalized Lebesgue
space $L^{p(\cdot)}(\partial\Omega)$
$$
 L^{p(\cdot)}(\partial\Omega) = \left\{f \in L^1(\dom) : \int_{\dom} |\lambda
f(z)|^{p(z)}d\sigma<\infty\ \ \rm{for\ some}\ \lambda>0  \right\}
$$
that with the Luxemburg norm
$$
\|f\|_{L^{p(\cdot)}(\partial\Omega)} = \inf \left\{\lambda>0 : \int_{\dom} \left|
\frac{f(z)}{\lambda} \right|^{p(z)}d\sigma(z)\leq 1\right\}
$$
is a Banach space.

Suppose $\Omega\subset\C^n,$ $n\geq2,$ is a bounded domain with $C^2-$boundary and $q:\partial\Omega\times\Omega\to\C^n$ is a $C^1-$map that satisfies $\scp{q(\xi,z)}{\xi-z}\neq 0$ for every $\xi\in\partial\Omega$ and $z\in\Omega$ then for any function $f$ holomorphic in $\Omega$ and
continuous in $\overline{\Omega}$ the following representation
formula holds
\begin{equation}
 f(z) = Kf(z)  = \frac{1}{(2\pi i)^n} \int_{\dom}
  \frac{f(\xi) q\wedge(\bar{\partial}_\xi q)^{n-1}}{\scp{q}{\xi-z}^n},\quad
  z\in\Omega.
\end{equation}

For the complex ellipsoid
$\Omega = \mathcal{E}$ we choose $q=\partial\rho,$ where $\rho(z) = \abs{z_1}^{2\alpha_1}+\ldots + \abs{z_n}^{2\alpha_n}.$ The real part of $\scp{\partial\rho(\xi)}{\xi-z}$ is equivalent to the Euclidean distance from $z\in \C$ to the real tangent plane at point $\xi\in\partial \mathcal{E}$ and, consequently, strictly positive if $z\in \mathcal{E}.$ Moreover the form $c_n
\partial\rho(\xi)\wedge(\bar{\partial}\partial\rho(\xi))^{n-1}$ defines a positive measure $dS$ such that $dS\sim \omega d\sigma,$ where
$\omega(\xi)=\prod\limits_{j = 1}^n m_j^2|\xi_j|^{2(m_j-1)}$ and $d\sigma$ is the induced Lebesgue measure on $\partial \mathcal{E}$. 
%We will also consider the volume measure $dV$ generated by the form $ (\bar{\partial}\partial\rho)^n = d\left( 
%\partial\rho(\xi)\wedge(\bar{\partial}\partial\rho(\xi))^{n-1}\right)$ which is equivalent to the Lebesgue measure $d\lambda$, that is $dV\sim \omega d\lambda.$ 
Consequently, 
\begin{equation}\label{eq1}
	 Kf(z) = \frac{1}{(2\pi i)^n} \int_{\partial\mathcal{E}} \frac{f(\xi) dS(\xi)}{w(\xi,z)^n},
\end{equation}
where
\begin{equation}
w(\xi,z) = \sum\limits_{j=1}^n m_j \abs{\xi_j}^{2(m_j-1)} \overline{\xi}_j (\xi_j-z_j).
\end{equation}

The operator $K$ is a singular integral
operator with a Calder\`{o}n--Zygmund kernel (see \cite{HA99}). 
In view of this result the operator $K$ is
bounded on the Lebesgue space $\Lp$ if the exponent $p$ is logarithmic
continuous and $1<p<\infty$ on $\partial \mathcal{E}$ (see \cite{DHHR11} for details
about Calder\`{o}n--Zygmund operators in variable-exponent spaces).

\begin{thm}
 Suppose $\mathcal{E}$ is a complex ellipsoid and $p:\partial \mathcal{E}\to (1,\infty)$
 is a logarithmic continuous (or logarithmic H\"{o}lder) function, that
 is
\begin{equation} \label{logCont}
  |p(z)-p(w)|\lesssim\frac{1}{ |\log{|z-w|}| },\ \textit{for}\ z,w\in\partial \mathcal{E},
\end{equation}
 then the operator $K$ in \eqref{eq1} maps $\Lp$ continuously into itself.
 \end{thm}
 
Remark that the condition \ref{logCont} is equivalent to the logarithmic continuity with respect to the quasimetric 
$$d(z,w) = \abs{\scp{\partial\rho(\xi)}{\xi-z}} + \abs{ \scp{\partial\rho(z)}{z-\xi} }  $$
since for $\delta = \max(m_1,\ldots,m_n)$ we have
$$ \abs{z-w}^{2\delta} \lesssim d(z,w) \lesssim \abs{z-w}.$$

%The construction is as follows: firstly,  we find a lower bound of these
%operators on characteristic functions of special sets, then we choose
%the appropriate linear combination of these characteristic
%functions such that this combination is $\Lp$-function and its transform under the operator $K_n$ is not in $\Lp.$

\section{Estimates of the kernel of the Cauchy--Leray--Fantappi\`{e} integral operator and construction of the test functions}

Following ideas of \cite{R17} we construct an example of
such function $p(\cdot)$ that operator $K$ is not bounded on
corresponding Lebesgue spaces $\Lp$. We construct an example for $n=2$;
the case $n>2$ is analogous. 

Let $m_1,m_2\geq 1.$ Consider an ellipsoid 
$$\mathcal{E} = \{z=(z_1,z_2)\in\C^2: \abs{z_1}^{2m_1}+\abs{z_2}^{2m_2} <1\}.$$
So 
\begin{equation*}%\label{eq1}
	Kf(z) = -\frac{1}{ 4\pi^2 } \int_{\partial\mathcal{E}} \frac{f(\xi) dS(\xi)}{w(\xi,z)^2},
\end{equation*}
where
\begin{equation*}
	w(\xi,z) = m_1 \abs{\xi_1}^{2(m_1-1)} \overline{\xi}_1 (\xi_1-z_1) + m_2 \abs{\xi_2}^{2(m_2-1)} \overline{\xi}_2 (\xi_2-z_2).
\end{equation*}

\begin{defn}[Boundary patches.]
	 Fix a small parameter $\alpha>0$ and parameters $0<\gamma_1<\gamma_2$, $0<\gamma_3<\gamma_4.$ Define source and target patches:
	 \begin{eqnarray}
	 	W_\alpha  =\left\{\xi=\left(r_1 e^{i \theta_1}, r_2 e^{i \theta_2}\right) \in \partial \mathcal{E}: -\frac{\pi}{12}\leq\theta_1\leq 0,\ -2\alpha\leq\theta_2 \leq -  \alpha,\ \gamma_1\alpha \leq r_1^{2m_1} \leq \gamma_2\alpha\right\}, \\
	 	V_\alpha  =\left\{z=\left(s_1 e^{i \varphi_1}, s_2 e^{i \varphi_2}\right) \in \partial \mathcal{E}: 0\leq\varphi_1\leq \frac{\pi}{12},\ \alpha\leq\varphi_2 \leq 2  \alpha,\  \gamma_3\alpha \leq r_1^{2m_1} \leq \gamma_4\alpha\right\},
	 \end{eqnarray}
	 and test functions 
	 $$h_\alpha(\xi) = \chi_{W_\alpha(\xi)},\quad H_\alpha(z) = Kh_\alpha(z).$$
\end{defn}

Geometrically, $W_\alpha$ and $V_\alpha$ represent thin angular sectors on opposite sides of the real axis, situated near the non-strictly convex portion of the boundary.

\begin{lem}[Measure of boundary patches] For every $0<\gamma_1<\gamma_2$ and $0<\gamma_3<\gamma_4$
	$$S(V_\alpha)\asymp S(W_\alpha) \asymp \alpha^{2}.$$
\end{lem}

\begin{proof}
		Consider a real-vector parametrization 
		$$
		\Phi\left(r_1, \theta_1, \theta_2\right)=\left(r_1 \cos \theta_1, r_1 \sin \theta_1, r_2\left(r_1\right) \cos \theta_2, r_2\left(r_1\right) \sin \theta_2\right),
		$$
		where
		$$
		r_2(r_1) =\left(1-r_1^{2m_1}\right)^{1/{(2m_2)}} = 1 - \frac{1}{2m_2}r_1^{2m_1} +O(r_1^{4m_1}).
		$$
		Then $V_\alpha $ is the image under $\Phi$ of the rectangle
		$$
		\left[\gamma_1 r_1^{1/(2m_1)},\gamma_2 r_1^{1/(2m_1)}\right] \times \left[0,\frac{\pi}{12}\right]\times [\alpha,2\alpha].
		$$
		Then the surface Lebesgue-measure $d\sigma$ equals the $3-$dimensional volume element induced by $\Phi,$ i.e.
		$$d\sigma(\xi) = \left\|\partial_{r_1} \Phi \wedge \partial_{\theta_1} \Phi \wedge \partial_{\theta_2} \Phi\right\| d r_1 d \theta_1 d \theta_2 $$,
		where the norm is the Euclidean 3-volume of the parallelepiped in  $\mathbb{R}^4.$
		
		Compute the partial derivatives
		\begin{eqnarray}
				\partial_{r_1} \Phi=\left(\cos \theta_1, \sin \theta_1, r_2^{\prime}\left(r_1\right) \cos \theta_2, r_2^{\prime}\left(r_1\right) \sin \theta_2\right), \\
				\partial_{\theta_1} \Phi=\left(-r_1 \sin \theta_1, r_1 \cos \theta_1, 0,0\right), \\
				\partial_{\theta_2} \Phi=\left(0,0,-r_2 \sin \theta_2, r_2 \cos \theta_2\right) .
		\end{eqnarray}
		These vectors are pairwise orthogonal, hence, 
		$$\left\|\partial_{r_1} \Phi \wedge \partial_{\theta_1} \Phi \wedge \partial_{\theta_2} \Phi\right\| = \left\|\partial_{r_1} \Phi\right\| \cdot\left\|\partial_{\theta_1} \Phi\right\| \cdot\left\|\partial_{\theta_2} \Phi\right\| = r_1 r_2\sqrt{1+\left(r_2'(r_1)\right)^2}.
		$$
		Notice that
		$$  1- \frac{\gamma_1}{2m_2}\alpha \leq r_2 \leq 1$$
		and
		$$r_2'(r_1) = \frac{m_1}{m_2}r_1^{2m_1-1} \left(1-r_1^{2m_1}\right)^{\frac{2m_2-1}{2m_2}} \asymp \alpha^{\frac{2m_1-1}{2m_1}}.
		$$
		Hence,
		$$ r_2\sqrt{1+\left(r_2'(r_1)\right)^2} \asymp 1$$
		and
		\begin{multline*}
		S\left(V_\alpha\right)=\int_{W_\alpha} m_1m_2 \abs{\xi_1}^{2(m_1-1)}\abs{\xi_2}^{2(m_2-1)} d\sigma(\xi)=\\
		m_1m_2\int_0^{\pi/12} \int_{(\gamma_1 \alpha)^{1/(2m_1)}}^{(\gamma_2 \alpha)^{1/(2m_1)}}
		r_1^{2m_1-1} r_2\left(r_1\right)^{2m_2-1} \sqrt{1+\left(r_2^{\prime}\left(r_1\right)\right)^2} d r_1 d \theta_1 d \theta_2 \asymp\\ \int_0^{\pi/12} \int_{\alpha}^{2 \alpha} \int_{(\gamma_1 \alpha)^{1/(2m_1)}}^{(\gamma_2 \alpha)^{1/(2m_1)}} r_1^{2m_1-1} d r_1 d \theta_1 d \theta_2 \asymp \alpha^{2}.
		\end{multline*}
		The same estimate clearly holds for $S(W_\alpha)\asymp \alpha^2.$		\qed
\end{proof}

It turns out that $w(\xi,z)$ is mostly imaginary when $\xi\in W_\alpha$ and $z\in V_\alpha.$ In fact we have the following lemma.

\begin{lem}[Estimate of the kernel on boundary patches] \label{lm:EstOfKernel}
	There exist parameters $0<\gamma_1<\gamma_2$, $0<\gamma_3<\gamma_4$ such that for sufficiently small $\alpha>0$ 
	\begin{equation}\label{lm:EstOfKernel_eq1}
		-\RE\left(\frac{1}{w(\xi,z)^2}\right) \gtrsim \frac{1}{\alpha^2}
	\end{equation}
	uniformly for $\xi\in W_\alpha$ and $z\in V_\alpha$. Moreover, we can choose parameters $0<\gamma_1<\gamma_2$, $0<\gamma_3<\gamma_4$ such that
	\begin{equation}\label{lm:EstOfKernel_eq2}
		-\RE\left(\frac{1}{w(\xi,z)^2}\right) > 0 , 
	\end{equation}
	for $\xi\in W_{\alpha_1}$ and $z\in V_{\alpha_2}$ if $\alpha_1, \alpha_2>0$ are small enough.
\end{lem}

\begin{proof} 
	Indeed, let $\xi=(\xi_1,\xi_2) =(r_1e^{i\theta_1},r_2e^{i\theta_2}) \in W_\alpha$ and $z = (s_1e^{i\varphi_1},s_2e^{i\varphi_2}) \in V_\alpha$. Then
	$$w(\xi,z) = m_1r_1^{2m_1-1}\left(r_1-s_1e^{i(\varphi_1-\theta_1)}\right) + m_2r_2^{2m_2-1}\left(r_2-s_2e^{i(\varphi_2-\theta_2)}\right)  =  A_1 + A_2.$$
	
	Since
	$$\abs{ r_1-s_1e^{i(\varphi_1-\theta_1)} }\leq \abs{r_1}+\abs{s_1} \lesssim \alpha^{1/(2m_1)}
	$$
	then
	$$\abs{A_1}\lesssim r_1^{2m_1-1}\alpha^{1/(2m_1)} \lesssim \alpha.$$
	To estimate $A_2$ notice that
	$$\abs{r_2-s_2}=(1-r_1^{2m_1})^{1/(2m_2)} - (1-s_1^{2m_1})^{1/(2m_2)} \lesssim \abs{ r_1^{2m_1} -s_1^{2m_1} } \lesssim \alpha.$$
	Hence, since $\varphi_2-\theta_2\in (2\alpha,4\alpha)$, we conclude 
	$$A_2 \lesssim \abs{r_2-s_2}+s_2\abs{1-e^{i(\varphi_2-\theta_2)} }\lesssim \alpha.$$
	Combing these estimates we see that 
	$$\abs{w(\xi,z)} \lesssim \alpha.$$
	
	The estimates above are valid for arbitrary $\gamma_j>0,$ $j=1,\ldots4.$ However, we also need to control the argument of $w(\xi,z).$ This will be achieved by the appropriate choice of parameters $\gamma_j,$ $j=1,\ldots4.$ Let first $\gamma_4=\frac{\gamma_1}{2}$ and $\gamma_{3} =\frac{\gamma_{4}}{2}$ and recall that
	$$ \gamma_1\alpha \leq r_1^{2m_1} \leq \gamma_2\alpha; \quad  \gamma_3\alpha \leq r_1^{2m_1} \leq \gamma_4\alpha ;$$
	$$r_2^{2m_2} = 1-r_1^{2m_1};\quad s_2^{2m_2} =  1-s_1^{2m_1}.$$

	The choice of $\gamma_4=\frac{\gamma_1}{2}$ ensures that $2s_1^{m_1}\leq  r_1^{m_1}<1$ and choosing $\alpha>0$ small enough we may assume that $2^{-1/(2m_2)}<r_2<s_2< 1.$

	Consider first the real part of $w(\xi,z)$:
	\begin{multline} \label{lm:EstOfKernel_REw}
		\RE w(\xi,z) = m_1 r_1^{2m_1-1} \left(r_1-s_1\cos(\varphi_1-\theta_1)\right)+\\
		m_2 r_2^{2m_2-1} \left(r_2-s_2\cos(\varphi_2-\theta_2)\right) = B_1+B_2.
	\end{multline}
	The first term 
	$$B_1\leq m_1 r_1^{2m_1} \leq m_1 \gamma_2 \alpha.$$
	For the the second term notice that if $\alpha>0$ is small enough
	\begin{equation} \label{eq2}
		0<s_2-r_2= \frac{ s_2^{2m_2}-r_2^{2m_2} }{ \sum\limits_{k=0}^{2m_2} s_2^k r_2^{2m_2-k} } \leq 
		\frac{ r_1^{2m_1} - s_1^{2m_1} }{ 2m_2 r_2^{2m_2} } \leq \frac{ \gamma_2 -\gamma_4 }{ 2 m_2 }  \alpha + O(\alpha^2)
		\leq\frac{ \gamma_2 }{ m_2 } \alpha.
	\end{equation}
	
	Since $$1-\cos(\varphi_2-\theta_2) \leq 8\alpha^2$$ then 
	$$|B_2|\leq m_2 \abs{r_2-s_2} + m_2 s_2(1-\cos(\varphi_2-\theta_2))\leq {2\gamma_2}\alpha$$
	if $\alpha>0$ is small enough.
	Finally,
	$$\RE{w(\xi,z)} \leq (2+m_1) \gamma_2\alpha.$$
	The real part is, indeed, positive. 
	Analogously to \eqref{eq2} we see that
	\begin{equation} \label{eq3}
	s_2-r_2= \frac{ s_2^{2m_2}-r_2^{2m_2} }{ \sum\limits_{k=0}^{2m_2} s_2^k r_2^{2m_2-k} } \leq \frac{\gamma_2-\gamma_4}{2m_2} \alpha + O(\alpha^2)
	\end{equation}
	Hence,
	\begin{multline*}
		\RE{w(\xi,z)} = m_1 r_1^{2m_1} \left(1-\frac{s_1}{r_1}\cos(\varphi_1-\theta_1)\right)+\\
		m_2 r_2^{2m_2-1}  \left(r_2-s_2\cos(\varphi_2-\theta_2)\right) \geq
		m_1 r_1^{2m_1} \left(1-\frac{s_1}{2r_1}\right) -
		m_2  \left(s_2-r_2\right) \geq\\
		 m_1\left(1-\frac{1}{4}\right)\gamma_1 \alpha -   \frac{\gamma_{2}-\gamma_4}{2} \alpha +O(\alpha^2)
		 \leq \left(\frac{3}{4} m_1-\frac{\gamma_2}{2\gamma_1}\right)\gamma_1\alpha
		  = c(\gamma_2/\gamma_2)\gamma_1\alpha
	\end{multline*}
 	when $\alpha>0$ is small enough. Since $m_1\geq1$ then 
	we can choose the relation $\gamma_2/\gamma_1>0$ such that $c(\gamma_1,\gamma_2)>0$, which ensures that 
	$\RE w(\xi,z) >0$. %when $\alpha>0$ is small enough. 
	
	Consider now the imaginary part:
	\begin{multline}\label{eq:est:ImW}
		-\IM{w(\xi,z)} = m_1 r_1^{2m_1-1} s_1 \sin(\varphi_1-\theta_1) + 
		m_2 r_2^{2m_2-1} s_2 \sin(\varphi_2-\theta_2) > \\
		m_2 r_2^{2m_2-1} s_2 \sin(\varphi_2-\theta_2)\geq  \frac{m_2}{\pi}(\varphi_2-\theta_2) \geq\frac{2m_2}{\pi} \alpha.
	\end{multline}
	
	Hence,
	$$0< -\cot\arg w(\xi,z) =-\frac{\RE w(\xi,z)}{\IM w(\xi,z)} \leq \frac{2+m_1}{2m_2} \pi \gamma_2 < \frac{2+m_1}{2m_2} \pi \gamma_1  < \frac{1}{\sqrt{3} }
	$$
	if we choose $\gamma_1>0$ small enough. 
	
	Since $\RE{w(\xi,z)}>0$ and $\IM {w(\xi,z)}<0$ this guarantees $-\frac{\pi}{2}<\arg w(\xi,z)<-\frac{2\pi}{3}.$ Hence, $-\cos(2\arg w(\xi,z))>\frac{1}{2}$ and
	$$
	-\RE \frac{1}{w(\xi,z)^2} =-\frac{\cos(2\arg w(\xi,z))}{\abs{w(\xi,z)}^2} \gtrsim \frac{1}{\alpha^2}
	$$
	To prove that 
	$$-\RE\left(\frac{1}{w(\xi,z)^2}\right) \geq 0 , $$
	for $\xi\in W_{\alpha_1}$ and $z\in V_{\alpha_2}$ if $\alpha_1, \alpha_2>0$ are small enough we show that 
	$$\abs{\RE{w(\xi,z)}} \leq C(\alpha_1+\alpha_2) \leq \abs{\IM w(\xi,z)}.$$
	The estimate \eqref{eq:est:ImW} implies that
	$$ \abs{\IM w(\xi,z)} \geq \frac{m_2}{\pi}(\alpha_1+\alpha_2).$$
	
	%Suppose now $\alpha_1,\alpha_2<\alpha_0.$ 
	To estimate real part recall the formula \eqref{lm:EstOfKernel_REw}. Now
%	$$ 
%	\abs{r_1-s_1}\leq (\gamma_{2}\alpha_1)^{1/(2m_1)}+(\gamma_4\alpha_2)^{1/(2m_1)} <  2\gamma_{2}^{1/(2m_1)}(\alpha_1+\alpha_2)^{1/(2m_1)} ;
%	$$
	\begin{multline*}
	\abs{B_1}=r_1^{2m_1-1}\abs{ r_1-s_1\cos(\varphi_1-\theta_1) } \leq r_1^{2m_1} + r_1^{2m_1-1} s_1 \leq \\	\left( \gamma_2\alpha_1+(\gamma_2\alpha_1)^{ 1-1/(2m_1) }(\gamma_{4}\alpha_2)^{ 1/(2m_1) }\right)\leq 2\gamma_2(\alpha_1+\alpha_2)
	\end{multline*}
	To estimate the second term $B_2$ consider
	$$ 
	1-\cos(\varphi_2-\theta_2) = O((\alpha_1+\alpha_2)^2);
	$$
	$$
	r_2=\left(1-r_1^{2m_1}\right)^{1/(2m_2)} = 1-\frac{r_1^{2m_1}}{2m_2} + O(r_1^{4m_1}) = 1-\frac{r_1^{2m_1}}{2m_2} + O((\alpha_1+\alpha_2)^2);
	$$
	$$
	s_2 = 1-\frac{s_1^{2m_1}}{2m_2} + O((\alpha_1+\alpha_2)^2);
	$$
	Hence,
	\begin{multline*}
		\abs{s_2-r_2} =  \frac{ \abs{s_1^{2m_1}-r_1^{2m_1}} }{  2m_2 } + O(\alpha_1+\alpha_2)^2 \leq \\
		\frac{\gamma_2 + \gamma_4}{2m_2}(\alpha_1+\alpha_2)+O(\alpha_1+\alpha_2)^2\leq \frac{\gamma_2}{m_2} (\alpha_1+\alpha_2)
	\end{multline*}
	and
	$$\abs{B_2}\leq m_2\abs{s_2-r_2}+m_2(1-\cos(\varphi_2-\theta_2))\leq 
	\gamma_2 (\alpha_1+\alpha_2) +O(\alpha_1+\alpha_2)^2\leq 2 \gamma_2 (\alpha_1+\alpha_2)
	$$
	if $\alpha_0>0$ is small enough. 	
	
	Finally,
 	$$
		\abs{\RE w(\xi,z) }\leq \abs{B_1}+\abs{B_2} \leq 2\left(m_1+1\right)\gamma_{2}(\alpha_1+\alpha_2)
	$$
	Choosing $\gamma_{2}>0$ such that
	$$ 2\left(m_1+1\right)\gamma_{2} < \frac{m_2}{\pi}$$
	we have 
	$$\abs{\RE w(\xi,z)} < \abs{\IM w(\xi,z)}$$
	and
	$$-\RE \frac{1}{w(\xi,z)^2} = \frac{(\IM{w(\xi,z)})^2-(\RE{w(\xi,z)})^2}{\abs{w(\xi,z)}^4} >0$$
	for $\xi\in W_{\alpha_1}$ and $z\in V_{\alpha_1}$ 
	which finishes the proof of the Lemma.\qed	
\end{proof}

\begin{lem}[Estimate of the image of test functions]
	Suppose that $\gamma_j,$ $j=1,\ldots,4$ are chosen as in Lemma~\ref{lm:EstOfKernel}. Then
	$$\RE{H_\alpha}(z) \gtrsim 1,\quad z\in V_\alpha$$ 
	and 
	$$ 
	\RE{H_{\alpha_1}(z)} > 0,\quad z\in V_{\alpha_2}
	$$
	if $\alpha_1,\alpha_2$ are small enough. 
\end{lem}

\begin{proof}
	Indeed,
	$$\RE \frac{1}{w(\xi,z)^2} \gtrsim \frac{1}{\alpha^2},\ \xi\in W_\alpha,\ z\in V_\alpha.$$
	Consequently,
	$$ \RE{H_\alpha} = -\frac{1}{\pi^2} \RE \int_{W_\alpha} \frac{ dS(\xi) }{w(\xi,z)^2} \gtrsim \alpha^{-2}\alpha^{2} = 1.$$
	Also, if if $\alpha_1,\alpha_2$ are small enough then
	$$\RE \frac{1}{w(\xi,z)^2} > 0,\ \xi\in W_{\alpha_1},\ z\in V_{\alpha_2}$$
	and
	$$ 
	\RE{H_{\alpha_1}(z)} > 0,\quad z\in V_{\alpha_2}
	$$	
	which finishes the proof.\qed
\end{proof}

\section{The construction of the exponent such that the operator $K$ is not bounded in  $L^{p(\cdot)}(\partial \mathcal{E}).$}

Let $p_0>1$ and $\psi:[-\pi,\pi]\to(1-p_0,+\infty)$ be a continuous function such that it is increasing on $[0,\pi]$ and
\begin{eqnarray}
	\psi(-\pi)=\psi(\pi);\label{eq:exp1}\\ 
	\psi(\alpha)<0,\ \alpha<0;\ \psi(0)=0;\label{eq:exp2}\\
	\lim_{\alpha\to0+}\psi(\alpha)\ln\frac{1}{\alpha} = +\infty;\label{eq:exp3}
\end{eqnarray}
Then 
\begin{equation}
	\lim_{\alpha\to0+}\left(\frac{1}{p_0+\psi(\alpha)} - \frac{1}{p_0}\right) \ln\frac{1}{\alpha} = -\infty
\end{equation}
and
\begin{equation}
	\lim_{\alpha\to0+} \left(\frac{1}{\alpha} \right)^{\frac{1}{p_0+\psi(\alpha)} - \frac{1}{p_0}} = 0;
\end{equation}

So, we can choose $\alpha_k$ such that
\begin{equation} \label{eq:alpha_k:1}
	\alpha_{k+1}\leq 2\alpha_k
\end{equation}
and
\begin{equation}\label{eq:alpha_k:2}
	 \left( \frac{1}{\alpha_k} \right)^{ \frac{1}{p_0+\psi(\alpha_k)}-\frac{1}{p_0} } \leq 2^{-k/p_0};
\end{equation}
We denote 
$$W_k=W_{\alpha_k};\quad V_k=V_{\alpha_k};\quad
h_k=h_{W_{\alpha_k}};\quad H_k=H_{W_{\alpha_k}}.$$

Let 
$$\lambda_k= \alpha_{k}^{-\frac{2}{p_0+\psi(\alpha_{k})}}$$
and
$$h(\xi) = \sum\limits_{k=1}^\infty \lambda_k h_k(\xi).$$

For $\xi=\left(r_1e^{i\theta_1},r_2e^{i\theta_2}\right)\in\partial\mathcal{E}$ define 
$$ p(\xi) = p_0+\psi(\theta_2).$$

\begin{thm}
	The function $h$ belongs to $\Lp$ while its image $Kh$ doesn't.
\end{thm}

\begin{proof}
	Notice that since $\alpha_{k+1}<2\alpha_k$ the patches $W_k$, $V_k$ are disjoint. Hence,
	\begin{multline*}
	\int_{\partial\mathcal{E}} \abs{h(\xi)}^{p(\xi)} dS(\xi) \leq
	\sum\limits_{k=1}^\infty \int_{W_k}	\lambda_k^{p(\xi)} dS(\xi) \leq\\ \sum\limits_{k=1}^\infty \int_{W_k}	\lambda_k^{p_0} dS(\xi) \lesssim 
	\sum\limits_{k=1}^\infty \lambda_k^{p_0} S(W_k)\lesssim \sum\limits_{k=1}^\infty \alpha_k^{- 2p_0/(p_0+\psi(\alpha_k))}  \alpha_{k}^{2} \lesssim\\ \sum\limits_{k=1}^\infty \alpha_k^{-2{p_0}\left( \frac{1}{p_0}-\frac{1}{p_0+\psi(\alpha_k)} \right)}\lesssim 
	\sum\limits_{k=1}^\infty 2^{- k}<\infty.
	\end{multline*}
	Consequently, $h\in \Lp.$
	
	Let $H(z)=Kh(z).$ Then 
	$$\abs{H(z)}\geq \RE{H(z)} \geq \sum\limits_{j = 1}^\infty \lambda_j \RE{H_j(z)} > \lambda_k \RE{H_k(z)} = \lambda_k ,\quad z\in V_k,$$
	since $\RE{H_j(z)}>0$ if $z\in V_k$. Consequently, 
	\begin{multline*}
		\int_{\partial\mathcal{E}} \abs{H(z)}^{p(z)} dS(z) \geq 
		\sum\limits_{k=1}^\infty \int_{V_k} \abs{H(z)}^{p(z)} dS(z)>\\
		\sum\limits_{k=1}^\infty \int_{V_k} \lambda_k^{p_0+\psi(z)} dS(z)\gtrsim \sum\limits_{k=1}^\infty \lambda_k^{p_0+\psi(\alpha_k)}\alpha_k^2 =\sum\limits_{k=1}^\infty  1 =+ \infty
	\end{multline*} 
	and $H\not\in \Lp.$ This finishes the proof of the theorem.\qed
\end{proof}

%\section{Concluding remarks}
%
%\begin{itemize}
%	\item The same construction applies to higher-dimensional ellipsoids
%	$$\mathcal{E}=\left\{z\in\Cn:\sum_{j=1}^n\left|z_j\right|^{2 m_j}<1\right\}$$
%	\item Instead of the conditions $\theta_1,\varphi_1\in\left(0,\pi/12\right)$ in the definitions of $W_\alpha,$ $V_\alpha$ we can  $\theta_1,\varphi_1\in\left(0,\delta(\alpha)\right)$, where $\delta(\alpha)=o(1)$  is such that 
%	$$\sum_{k=1}^\infty \delta(\alpha_k)^{p_0+\psi(\alpha_k)} = +\infty,$$
%	where $\alpha_k$ are choose according to conditions \eqref{eq:alpha_k:1} and \eqref{eq:alpha_k:2}. 
%\end{itemize}

%\section*{References}

\end{document}